\documentclass[11pt]{article} 
\usepackage{a4,amssymb, amsmath, amsthm, color}

\newtheorem{thm}{Theorem}[section]

\newtheorem{lem}[thm]{Lemma}
\newtheorem{pro}[thm]{Proposition}

\theoremstyle{definition}

\theoremstyle{remark}
\newtheorem{rem}[thm]{Remark}

\newcommand{\C}{\mathbb{C}}
\newcommand{\R}{\mathbb{R}}
\newcommand{\Rn}{{\mathbb{R}^n}}

\newcommand{\om}{\Omega }
\newcommand{\pom}{{\partial \Omega} }
\newcommand{\omc}{ {\Omega^c} }
\newcommand{\gs}{{\Gamma_s}}
\newcommand{\gt}{{\Gamma_t}}
\newcommand{\Wgs}{{\widetilde{W}^{\frac{1}{2},2}(\gs)}}
\newcommand{\Whgs}{{\widetilde{W}_h^{\frac{1}{2},2}(\gs)}}
\newcommand{\Wg}{{W^{\frac{1}{2},2}(\pom)}}
\newcommand{\Whmg}{{W^{-\frac{1}{2},2}_h(\pom)}}
\newcommand{\Wmg}{{W^{-\frac{1}{2},2}(\pom)}}
\newcommand{\Wp}{{W^{1,p}(\om)}}
\newcommand{\Whp}{{W_h^{1,p}(\om)}}

\begin{document}

\title{Adaptive FE--BE Coupling for Strongly Nonlinear Transmission Problems with Friction II}
%\author{Heiko Gimperlein\thanks{Department of Mathematical Sciences, University of Copenhagen, Universitetsparken 5, 2100 Copenhagen \O , Denmark, email: gimperlein@math.ku.dk} \and Ernst P.~Stephan\thanks{Institute of Applied Mathematics, Leibniz University Hannover, Welfengarten 1, 30167 Hannover, Germany, email: stephan@ifam.uni-hannover.de }}
\author{Heiko Gimperlein\thanks{This work was supported by the Danish National Research Foundation (DNRF) through the Centre for Symmetry and Deformation and the Danish Science Foundation (FNU) through research grant 10-082866. H.~G.~thanks the Institut f\"{u}r Angewandte Mathematik for hospitality.} \and Ernst P.~Stephan}
\date{}

\maketitle \vskip 0.5cm
\begin{abstract}
\noindent This article discusses the well-posedness and error analysis of the coupling of finite and boundary elements for transmission or contact problems in nonlinear elasticity. It concerns ``pseudoplastic'', $p$--Laplacian-type Hencky materials with an unbounded stress--strain relation, as they arise in the modelling of ice sheets, non-Newtonian fluids or porous media. For $1<p<2$ the bilinear form of the boundary element method fails to be continuous in natural function spaces associated to the nonlinear operator. We propose a functional analytic framework for the numerical analysis and obtain a priori and a posteriori error estimates for Galerkin approximations to the resulting boundary/domain variational inequality. The a posteriori estimate complements recent estimates obtained for mixed finite element formulations of friction problems in linear elasticity.
\end{abstract}
\vskip 1.0cm

\section{Introduction}\label{intro}

Let $n=2$ or $3$ and $\Omega \subset \Rn$ be a bounded Lipschitz domain. We consider transmission and frictional contact problems between a nonlinear, uniformly $W^{1,p}(\Omega)$--monotone operator in $\Omega$ and the homogeneous Lam\'{e} equation in the exterior domain. Adaptive finite element / boundary element procedures provide an efficient and extensively investigated tool for the numerical solution when the nonlinear operator is uniformly elliptic \cite{stephan}. Their analysis, however, does not apply to the above ``pseudoplastic'' material laws arising in the modelling of ice sheets, non-Newtonian fluids or porous media \cite{fluid, ice}, because for $p<2$ the bilinear form of the boundary element method fails to be continuous on natural function spaces related to the nonlinear operator. This article provides a functional analytic framework to study the wellposedness and an error analysis of FE / BE coupling procedures in this situation.\\

\noindent \textbf{{Formulation of Problem:}} We consider the following contact problem for $(u,u_c) \in (W^{1,p}(\om))^n \times (W^{1,2}_{loc}(\omc))^n$, where $p\in(1,\infty)$ and $\partial\Omega = \overline{\Gamma_s \sqcup \Gamma_t}$ is decomposed into two open subsets:
\begin{eqnarray}\label{CTP}
-\mathrm{div}\, A'(\varepsilon(u))&=& f \quad \text{in $\om$,} \nonumber \\
- \mu \Delta u_c - (\lambda+\mu)\, \mathrm{grad}\, \mathrm{div}\ u_c &=& 0 \quad \text{in $\omc$,} \nonumber\\
A'(\varepsilon(u)) \nu -T^* u_c&=& t_0 \quad \text{on $\pom$,} \label{diff}\\
u-u_c&=& u_0\quad \text{on $\gt$,}\nonumber
\end{eqnarray}
and a radiation condition $u(x) = \emph{o}(1)$, $\mathrm{grad}\ u(x) =\mathcal{O}(|x|^{-1})$ resp.~$u(x) = \mathcal{O}(|x|^{-1})$, $\mathrm{grad}\ u(x) =\mathcal{O}(|x|^{-2})$ is satisfied for $n=2$, $3$ as $|x|\to \infty$.
On $\gs$ contact conditions corresponding to Tresca friction are imposed. If $\nu$ denotes the unit outer normal to $\pom$, the conditions are given in terms of the normal and tangential components of $u$, $u_n = \nu \cdot u$ and $u_t = u - u_n \nu$, and of the stress, $\sigma_n(u)= -\nu A'(\varepsilon(u))\nu$ and $\sigma_t(u) = -A'(\varepsilon(u)) \nu - \sigma_n(u) \nu$:
\begin{align*}& \sigma_n(u)\leq 0\ , \ u_{0,n}+u_{c,n}-u_n \leq 0\ , \ \sigma_n(u)(u_{0,n}+u_{c,n}-u_n)=0\ , \\
&|\sigma_t(u)|\leq \mathcal{F} \ , \ \sigma_t(u) (u_{0,t}+u_{c,t}-u_t)+\mathcal{F}|u_{0,t}+u_{c,t}-u_t|)=0\ .
\end{align*}
We have denoted the strains by $\varepsilon_{ij}(u) =\frac{1}{2}(\partial_{x_i} u_j + \partial_{x_j} u_i)$ and the natural
conormal derivative $2 \mu \partial_\nu + \lambda \nu \mathrm{div} + \mu \nu \times \mathrm{curl}$ at the boundary by $T^*$. The exterior problem is strongly elliptic provided $\mu>0$, $\lambda>-\mu$. The function $A' : L^p(\om) \otimes \R^{n\times n}_{\mathrm{sym}} \to L^{p'}(\om) \otimes \R^{n\times n}_{\mathrm{sym}}$ is assumed to be a bounded, continuous and uniformly monotone operator, so that in particular for $p\in(1, 2)$:
\begin{align}\label{upperbound1}
\langle A'(x)-A'(y), x-y \rangle &\gtrsim (\|x\|_{L^p(\om)}+ \|y\|_{L^p(\om)})^{p-2} \|x-y\|_{L^p(\om)}^2 \ ,\nonumber\\
\langle A'(x)-A'(y), z \rangle &\lesssim \|x-y\|_{L^p(\om)}^{p-1} \|z\|_{L^p(\om)}\ .
\end{align}
When $p\in[2,\infty)$, we require
\begin{align}\label{upperbound2}
\langle A'(x)-A'(y), x-y \rangle &\gtrsim \|x-y\|_{L^p(\om)}^p \ ,\nonumber\\
\langle A'(x)-A'(y), z \rangle &\lesssim (\|x\|_{L^p(\om)}+ \|y\|_{L^p(\om)})^{p-2}\|x-y\|_{L^p(\om)} \|z\|_{L^p(\om)}\ .
\end{align}
We assume $\gt\neq \emptyset$, the compatibility condition $\int_\Omega f + \langle t_0, 1 \rangle=0$ for $n=2$ and that the data belong to the following spaces:
$$f\in (L^{p'}(\Omega))^n,\,\, u_0 \in (W^{\frac{1}{2},2}(\pom))^n,\,\, t_0 \in (W^{-\frac{1}{2},2}(\pom))^n, \,\,0\leq \mathcal{F} \in L^\infty(\gs)\ .$$
In Theorem \ref{existence} we will show that Problem \eqref{CTP} admits a unique weak solution $(u_1,u_2) \in W^{1,p}(\om)^n \times W^{1,2}_{loc}(\omc)^n$.\\
Examples include, in particular, $p$-Laplacian materials with $A'(x) = |x|^{p-2}x$ as well as Carreau-type laws
$A'(x) = (|x|^{1-\delta}(1+|x|^2)^{\delta})^{\frac{p-2}{2}}x$ with $\delta \in [0,1]$. \\

For the symmetric coupling of finite and boundary elements, the Poincar\'{e}--Steklov operator $S$ of the Lam\'{e} equation on $\omc$ is used to reduce Problem \eqref{CTP} to a variational inequality in the Banach space
$$X^p=\{(u,v) \in (W^{1,p}(\om))^n \times (\widetilde{W}^{1-\frac{1}{r},r}(\gs))^n\  : u|_\pom+v\in W^{\frac{1}{2},2}(\pom)^n\}\ ,$$
where $r=\min\{p,2\}$. \\

\noindent \textbf{{Main Results:}} This article complements the analysis of \cite{scalar}, which concerned a scalar $p$--Laplacian-type problem with frictional contact in the simpler case of ``dilatant'' material laws with $2\leq p<\infty$. In \cite{scalar} numerical approximations of the variational inequality could be studied in $\widetilde{X}^p = (W^{1,p}(\om))^n \times (\widetilde{W}^{\frac{1}{2},2}(\gs))^n$, as $\widetilde{X}^p = X^p$ for $p\geq 2$, with an emphasis on the transmission problem. Numerical examples confirmed the theoretical estimates.\\

Here we show that the space $X^p$ provides the proper setting for the numerical analysis for all $p \in (1,\infty)$, and we focus on the more intricate wellposedness and a sharp error analysis of the friction problem when $p \in (1,2)$: While the a posteriori estimate in \cite{scalar} was aimed at the pure transmission problem, Theorem \ref{a posteriori} gives a sharp a posteriori estimate for the error of Galerkin approximations to the variational inequality. It complements recent results for mixed finite element formulations of friction problems \cite{hl,hw,schrod} and is new even in the elliptic case. \\
The existence of a unique $X^p$--solution is shown in Theorem \ref{existence}, and Theorem \ref{a priori} gives an a priori estimate for Galerkin approximations. Finally, in Section \ref{fulldisc} we sketch the analysis when the discretization of the Poincar\'{e}--Steklov operator is included. As an example of the added difficulty when $p \in (1,2)$, the variational inequality no longer splits into an equality on $\Omega$ and an inequality on $\pom$, unless the artificial regularity assumption $u|_\pom \in W^{\frac{1}{2},2}(\pom)^2$ is imposed.\\

The results in this article are stated for $p \in (1, \infty)$, but we refer to \cite{scalar} for most of the arguments when $p\geq 2$. Conversely, an appendix adapts the new a posteriori estimate for the frictional term to the setting considered there. \\

The mathematical differences between $p<2$ and $p\geq2$ are not artificial. They reflect the different physical behavior: While pseudoplastic materials like ice or molasses ($p<2$) get stiffer and stiffer under a smaller stress, possibly infinitely so, the opposite happens in the dilatant case like a thick emulsion of sand and water ($p>2$).

\vspace*{0.2cm}

\section{Preliminaries}

Let $\om$ be a bounded, open subset of $\Rn$ with Lipschitz boundary $\pom$. Set $p'=\frac{p}{p-1}$ whenever $p \in (1,\infty)$. We will also denote $r=\min\{p,2\}$ and $q=\max\{p,2\}$.

Before analyzing a variational formulation of (\ref{TP}), we recall some properties of $L^p$--Sobolev spaces on $\om$:

\begin{rem} \label{Wsp}
a) $\left(W^{s,p}(\pom)\right)'=W^{-s,p'}(\pom)$ and $W^{s,2}(\pom) = H^s(\pom)$.\\
b) $W^{s,2}(\om) \hookrightarrow W^{s,p}(\om)$ and $\|u\|_{W^{s,p}(\om)} \leq |\om|^{1-\frac{2}{p}}\|u\|_{W^{s,2}(\om)}$ for $1<p\leq 2$.\\
c) If $\pom$ is smooth, pseudodifferential operators of order $m$ with $\C^{k\times \tilde{k}}$--valued symbol in the H\"ormander class $S^m_{1,0}(\pom)$ map $(W^{s,p}(\pom))^{k}$ continuously to $(W^{s-m,p}(\pom))^{\tilde k}$. For Lipschitz $\pom$, at least the first--order Steklov--Poincar\'e operator $S$ of the Lam\'e operator on $\omc$ is continuous between $(W^{\frac{1}{2},2}(\pom))^{n}$ and $(W^{-\frac{1}{2},2}(\pom))^n$. \\
d) Points a) to c) imply that the quadratic form $\langle S u, u \rangle$ associated to $S$ is well-defined on $(W^{1-\frac{1}{p},p}(\pom))^n$ if $p \geq 2$. $S$ being elliptic, the form is unbounded for $p<2$ even if $\pom$ is smooth.
\end{rem}

The fundamental solution for the Lam\'e operator in $\mathbb{R}^2$,
$$G(x,y) = \frac{\lambda + 3 \mu}{4 \pi \mu(\lambda + 2 \mu)} \left\{ \log(|x-y|^{-1})\ \mathrm{Id} + \frac{\lambda + \mu}{\lambda + 3 \mu} \frac{(x-y)(x-y)^T}{|x-y|^2}\right\}\ ,$$
resp.~$\mathbb{R}^3$
$$G(x,y) = \frac{\lambda + 3 \mu}{4 \pi \mu(\lambda + 2 \mu)} \left\{ \frac{1}{|x-y|}\ \mathrm{Id} + \frac{\lambda + \mu}{\lambda + 3 \mu} \frac{(x-y)(x-y)^T}{|x-y|^3}\right\}\ ,$$
allows to define layer potentials on $\pom$ associated to the exterior problem in the usual way:
\begin{eqnarray*}
\mathcal{V} \phi(x) &=& \int_\pom \phi(x')\ G(x,x')\ dx',\\
\mathcal{K} \phi(x) &=& \int_\pom \phi(x')\ \partial_{\nu_{x'}} G(x,x') \ dx',\\
\mathcal{K}'\phi(x) &=& \int_\pom \phi(x')\
\partial_{\nu_x} G(x,x')\ dx',\\
\mathcal{W} \phi(x) &=& \partial_{\nu_x}\int_\pom \phi(x')\ \partial_{\nu_{x'}} G(x,x') \ dx' \ .
\end{eqnarray*}
They extend from $C^\infty(\pom)^n$ to a bounded map $\begin{pmatrix} -\mathcal{K} & \mathcal{V}\\
\mathcal{W} & \mathcal{K}' \end{pmatrix}$ on the Sobolev space $ W^{\frac{1}{2},2}(\pom)^n \times
W^{-\frac{1}{2},2}(\pom)^n$. If (for $n=2$) the capacity of $\pom$ is less than
$1$, which can always be achieved by scaling, $\mathcal{V}$ and
$\mathcal{W}$ considered as operators on $W^{-\frac{1}{2},2}(\pom)^n$
are selfadjoint, $\mathcal{V}$ is positive and $\mathcal{W}$
non-negative. The Steklov-Poincar\'e operator for the exterior Lam\'{e} problem is given as
$$S = \mathcal{W}+(1-\mathcal{K}')\mathcal{V}^{-1}(1-\mathcal{K}): W^{\frac{1}{2},2}(\pom)^n \subset
W^{-\frac{1}{2},2}(\pom)^n \to W^{-\frac{1}{2},2}(\pom)^n$$ and defines a
positive and selfadjoint operator with the main property
$$T^* u_2|_\pom = - S (u_2|_\pom)$$ for solutions $u_2$ of the Lam\'{e} equation on $\omc$ satisfying the decay condition at $\infty$.
$S$ therefore gives rise to a coercive and symmetric bilinear form $\langle S u , u\rangle$ on $W^{\frac{1}{2},2}(\pom)^n$. \\

Existence of a unique solution to \eqref{CTP} will be shown using Korn's inequality and coercivity:
\begin{pro}{(\cite{scalar}, Proposition 2)}\label{korn}
Assume $\om \subset \Rn$ is a bounded Lipschitz domain and $\Gamma \subset \pom$ has positive $(n-1)$--dimensional measure. Then there is a $C>0$ such that $$\|u\|_{1,p} \leq C( \|\varepsilon(u)\|_p + \|u|_{\Gamma}\|_{L^1(\Gamma)}) \quad \text{for all $u \in (W^{1,p}(\om))^n$}.$$
\end{pro}

\vspace*{0.2cm}

\section{Analysis of the boundary integral formulation}

For $r=\min\{p,2\}$, we consider the space $$X^p=\{(u,v) \in (W^{1,p}(\om))^n \times (\widetilde{W}^{1-\frac{1}{r},r}(\gs))^n\  : u|_\pom+v\in W^{\frac{1}{2},2}(\pom)^n\}$$
equipped with the norm $$\|u,v\|_{X^p} = \|u\|_{W^{1,p}(\om)} + \|v\|_{\widetilde{W}^{1-\frac{1}{r},r}(\gs)} + \|u|_\pom + v\|_{W^{\frac{1}{2},2}(\pom)} \ .$$
Note that $X^p = (W^{1,p}(\om))^n \times (\widetilde{W}^{\frac{1}{2},2}(\gs))^n$ when $p\geq 2$, so that we recover a vector--valued variant of the Banach spaces considered in \cite{scalar}.

\begin{lem}\label{normequiv}
$(X^p, \|\cdot\|_{X^p})$ is a Banach space, and
$$|u,v|_{X^p} = \|u\|_{W^{1,p}(\om)} +  \|u|_\pom + v\|_{W^{\frac{1}{2},2}(\pom)}$$
defines an equivalent norm on $X^p$.
% The dual space $X'$ is
% $$X' = \{(u+\tau,v|_\gs-\tau,\phi) : u \in W_0^{-1,p'}(\om)^2,\ v \in W^{-\frac{1}{p'},p'}(\pom)^2,\ \tau \in \widetilde{W}^{-\frac{1}{2},2}(\gs)^2\ , \phi\in W^{\frac{1}{2},2}(\pom)^2\}$$
\end{lem}
\begin{proof}
It is readily verified that $\|\cdot \|_{X^p}$ defines a norm on $X^p$. To show completeness, let $(u_j, v_j) \in X$ be a Cauchy sequence. Then $(u_j, v_j)$ converges to a limit $(u,v)$ in the Banach space $W^{1,p}(\om)^n \times \widetilde{W}^{1-\frac{1}{r},r}(\gs)^n$. Also $u_j|_\pom + v_j$ converges to a limit $w$ in $W^{\frac{1}{2},2}(\pom)^n$. However, the continuity of the trace operator assures that $u_j|_\pom \to u|_\pom$ in ${W}^{1-\frac{1}{p},p}(\pom)^n$. Therefore in ${W}^{1-\frac{1}{p},p}(\pom)^n$, hence also in ${W}^{1-\frac{1}{r},r}(\pom)^n$, $u_j|_\pom + v_j$ converges both to $u|_\pom + v$ and to $w$. This means that $u|_\pom + v = w \in W^{\frac{1}{2},2}(\pom)^n$, or $(u, v) \in X^p$.\\

To see the equivalence of norms, note that $|u,v|_{X^p}\leq \|u,v\|_{X^p}$. On the other hand, the continuous inclusion of $W^{\frac{1}{2},2}(\pom)$ into ${W}^{1-\frac{1}{r},r}(\pom)$, of ${W}^{1-\frac{1}{p},p}(\pom)$ into ${W}^{1-\frac{1}{r},r}(\pom)$, and the continuity of the trace operator from $W^{1,p}(\om)$ to ${W}^{1-\frac{1}{p},p}(\pom)$
imply
\begin{align*}
\|u,v\|_{X^p} &\leq \|u\|_{W^{1,p}(\om)} + \|u|_\pom\|_{{W}^{1-\frac{1}{r},r}(\pom)}+\|u|_\pom + v\|_{{W}^{1-\frac{1}{r},r}(\pom)}\\
 & \qquad + \|u|_\pom + v\|_{W^{\frac{1}{2},2}(\pom)}\\
&\leq \|u\|_{W^{1,p}(\om)} + \|u|_\pom\|_{{W}^{1-\frac{1}{p},p}(\pom)}+\|u|_\pom + v\|_{{W}^{\frac{1}{2},2}(\pom)}\\
 & \qquad + \|u|_\pom + v\|_{W^{\frac{1}{2},2}(\pom)}\\
&\lesssim \|u\|_{W^{1,p}(\om)} + \|u|_\pom + v\|_{{W}^{\frac{1}{2},2}(\pom)}\\
& = |u,v|_{X^p} \ .
\end{align*}
The assertion follows.
\end{proof}

We consider a variational formulation of the contact problem in terms of the functional
$$J(u,v) = \langle A(\varepsilon(u)), \varepsilon(u)\rangle + \frac{1}{2}\langle S(u|_\pom+v), u|_\pom+v\rangle - L(u,v)$$
on $X^p$. Here $A$ is derived from $A'$ by an explicit formula, $v=u_0+u_c-u$, $$j(v) = \int_\gs  \mathcal{F}\ |v_t|\ ,$$ and $$L(u,v) = \int_\om f u + \langle t_0 + Su_0, u|_\pom+v \rangle\ .$$
This paper investigates the numerical approximation of the following nonsmooth variational problem over the closed convex subset $$K = \left\{(u,v) : v_n \leq 0, \langle S 1, u|_\pom +v-u_0\rangle = 0\right\}$$ of $X^p$: \\
\noindent Find $(\hat u,\hat v) \in K$ such that
\begin{equation}\label{minproblem}J(\hat u,\hat v) +j(\hat v)= \min_{(u,v) \in K} J(u,v)+j(v)\ .
\end{equation}
Note that $j$ is Lipschitz, but not differentiable.\\

As in \cite{scalar} one observes that Problem \eqref{minproblem} is equivalent to the contact problem \eqref{CTP}. The existence of a unique solution to the latter is therefore a consequence of the following theorem.

\begin{thm}\label{existence}
There exists a unique minimizer $(\hat u,\hat v) \in K$ of $J+j$ over $K$.
\end{thm}
The crucial ingredient in the proof is a monotonicity estimate:
\begin{lem}\label{monotony} The operator associated to $J$ is strongly monotone on $X^p$.\\
Let $r = \min\{p,2\}$, $q = \max\{p,2\}$ and $C>0$. Then for every $(u_1,v_1),(u_2,v_2) \in X^p$ with $\|u_{1},v_1\|_{X^{p}(\om)}, \|u_{2},v_2\|_{X^p(\om)}<C$, there holds
\begin{align*}
& \|u_2-u_1,v_2-v_1\|_{X^p}^q\\
&\qquad \lesssim_C \langle A'(\varepsilon(u_2))-A'(\varepsilon(u_1)), \varepsilon(u_2)-\varepsilon(u_1)\rangle\\
& \qquad \qquad+ \langle S (( u_2- u_1)|_\pom + v_2- v_1), (u_2-u_1)|_\pom + v_2-v_1\rangle\\
& \qquad\qquad \lesssim_C \|u_2-u_1,v_2-v_1\|_{X^p}^r \ .
\end{align*}
\end{lem}
\begin{proof} The upper bound is a consequence of the estimates \eqref{upperbound1}, \eqref{upperbound2} for the nonlinear operator and the boundedness of $S$ from $W^{\frac{1}{2},2}(\pom)^n$ to $W^{-\frac{1}{2},2}(\pom)^n$. For $p\geq 2$, we refer to \cite{scalar}, Lemma 3, for the proof of an analogous lower estimate.\\
When $p<2$ the monotony of $A'$ resp.~coercivity of $S$ imply for any $\delta \in (0,1)$
\begin{align}\label{estimate1}
&\langle A'(\varepsilon(u_2))-A'(\varepsilon(u_1)), \varepsilon(u_2)-\varepsilon(u_1)\rangle\nonumber\\
& + \langle S (( u_2- u_1)|_\pom + v_2- v_1), (u_2-u_1)|_\pom + v_2-v_1\rangle\nonumber\\
& \gtrsim \|\varepsilon(u_2-u_1)\|^p_{L^p(\om)} + \|(u_2-u_1)|_\pom + v_2-v_1\|_{W^{\frac{1}{2},2}(\pom)}^2 \nonumber\\
& \gtrsim \|\varepsilon(u_2-u_1)\|^2_{L^p(\om)} + \|(u_2-u_1)|_\pom + v_2-v_1\|_{W^{\frac{1}{2},2}(\gs)}^2 + \|u_2-u_1\|^2_{W^{\frac{1}{2},2}(\gt)}\nonumber\\
& \qquad +\|(u_2-u_1)|_\pom + v_2-v_1\|_{W^{\frac{1}{2},2}(\pom)}^2\nonumber\\
& \gtrsim \|\varepsilon(u_2-u_1)\|^2_{L^p(\om)} + \delta \|(u_2-u_1)|_\pom + v_2-v_1\|_{W^{1-\frac{1}{p},p}(\gs)}^2 + \|u_2-u_1\|^2_{W^{\frac{1}{2},2}(\gt)}\nonumber\\
& \qquad +\|(u_2-u_1)|_\pom + v_2-v_1\|_{W^{\frac{1}{2},2}(\pom)}^2 \ .
\end{align}
In the last inequality we use the continuous inclusion $W^{\frac{1}{2},2}(\gs) \subset W^{1-\frac{1}{p},p}(\gs)$.
Korn's inequality, Proposition \ref{korn}, implies
\begin{equation}\label{estimate2}
\|\varepsilon(u_2-u_1)\|^2_{L^p(\om)}+\|u_2-u_1\|^2_{W^{\frac{1}{2},2}(\gt)}\gtrsim \|u_2-u_1\|^2_{W^{1,p}(\om)}\ .\end{equation}
Further note from the triangle inequality, the convexity of $x \mapsto x^2$ as well as the continuity of the trace map from $W^{1,p}(\om)$ to $W^{1-\frac{1}{p},p}(\gs)$:
\begin{align}\label{estimate3}
& \|v_2- v_1\|^2_{\widetilde W^{1-\frac{1}{p},p}(\gs)} \leq \left(\|(u_2-u_1)|_\gs+v_2- v_1\|_{ W^{1-\frac{1}{p},p}(\gs)}+\|(u_2-\hat u_1)|_\gs\|_{ W^{1-\frac{1}{p},p}(\gs)}\right)^2 \nonumber\\
& \leq 2 \|(u_2- u_1)|_\gs+v_2- v_1\|_{ W^{1-\frac{1}{p},p}(\gs)}^2+2 \|u_2- u_1\|^2_{ W^{1-\frac{1}{p},p}(\gs)}\nonumber \\
& \leq 2 \|(u_2-u_1)|_\gs+v_2- v_1\|_{ W^{1-\frac{1}{p},p}(\gs)}^2+2C'\|u_2-u_1\|^2_{ W^{1,p}(\om)} \ .
\end{align}
The asserted estimate follows from \eqref{estimate1}, \eqref{estimate2} and \eqref{estimate3}, after choosing $\delta>0$ sufficiently small.\\
Strong monotony on all of $X^p$ is shown similarly, but for large $\|\varepsilon(u_2-u_1)\|_{ L^{p}(\om)}$ the exponent $2$ in the lower bound has to be replaced by $p$.
\end{proof}

\begin{proof}[Proof (of Theorem \ref{existence})]
By Lemma \ref{monotony} the operator associated to $J$ is bounded and strongly monotone. Existence and uniqueness for the perturbation $J+j$ of $J$ follow e.g.~by applying the perturbation result \cite{zei}, Proposition 32.36.
\end{proof}

\section{Discretization and a priori error analysis}

Let $\{\mathcal{T}_h\}_{h\in I}$ a regular triangulation of
$\om$ into disjoint open regular triangles ($n=2$) resp.~tetrahedra ($n=3$) $T$, so that
$\overline{\om} = \bigcup_{T \in \mathcal{T}_h} \overline{T}$. Each
element has at most one edge resp.~face on $\pom$, and the closures of any two
of them share at most a single vertex, edge or face. Let $h_T$ denote the
diameter of $T \in \mathcal{T}_h$ and $\rho_T$ the diameter of the
largest inscribed ball. We assume that $1 \leq \max_{T \in
\mathcal{T}_h} \frac{h_T}{\rho_T} \leq R$ independent of $h$ and
that $h = \max_{T\in \mathcal{T}_h} h_T$. $\mathcal{E}_h$ is going
to be the set of all edges of the triangles / faces of the tetrahedra in $\mathcal{T}_h$. Associated to $\mathcal{T}_h$ is the space $\Whp \subset
\Wp$ of functions whose restrictions to any $T \in \mathcal{T}_h$
are linear.

The boundary $\pom$ is triangulated by $\{l \in \mathcal{E}_h : l
\subset \pom\}$. For $r = \min\{p,2\}$, $W^{1-\frac{1}{r}, r}_h(\pom)$ denotes the corresponding space of continuous, piecewise linear functions, and $\widetilde{W}^{1-\frac{1}{r}, r}_h(\Gamma_s)$ the subspace of those supported on
$\gs$. Finally, $\Whmg \subset \Wmg$ is the space of piecewise constant functions, and $X^p_h = \Whp \times \Whgs \subset X^p$.

We denote by $i_h: \Whp \hookrightarrow \Wp$, $j_h : \Whgs\hookrightarrow
\Wgs$ and $k_h: \Whmg \hookrightarrow \Wmg$ the canonical inclusion
maps.

The discrete problem involves the discretized functional
\begin{equation*}
J_h(u_h,v_h) = \langle A(\varepsilon(u_h)), \varepsilon(u_h)\rangle + \frac{1}{2}\langle S(u_h|_\pom+v_h), u_h|_\pom+v_h\rangle - L_h(u_h,v_h)
\end{equation*}
on $X^p_h$. Here
\[
S_h = \frac12( W+(I-K')k_h(k_h^* V k_h)^{-1} kt_h^*(I-K))
\]
and $$L_h(u_h,v_h) =
\int_\om f u_h + \langle t_0+ S_h u_0, u_h|_\pom + v_h\rangle\ .$$
There exists $h_0>0$ such that the approximate Steklov--Poincar\'e operator $S_h$ is coercive
uniformly in $h<h_0$, i.e.~$\langle S_h u_h, u_h\rangle \geq
\alpha_S \|u_h\|_{\Wg}^2$ with $\alpha_S$ independent of $h$. Therefore, as in the previous section the discrete minimization problem
\begin{equation}\label{minproblemh}J(\hat u_h,\hat v_h) +j(\hat v_h)= \min_{(u_h,v_h) \in K \cap X_h^p} J(u_h,v_h)+j(v_h)\ .
\end{equation}
is associated to a perturbation of a strongly monotone operator on $X^p_h$ and admits a unique minimizer.\\

\noindent Our Galerkin method for the numerical approximation relies on an equivalent reformulation of the continuous and discretized minimization problems \eqref{minproblem}, \eqref{minproblemh} as variational inequalities:\\
\noindent Find $(\hat u, \hat v) \in K$ such that
\begin{align}\label{vi}
& \langle A'(\varepsilon(\hat u)), \varepsilon(u-\hat u)\rangle + \langle S(\hat u|_\pom+\hat v), (u-\hat u)|_\pom+v-\hat v\rangle\nonumber \\ & \qquad \qquad\qquad +j(v)-j(\hat v)\geq L(u-\hat u,v-\hat v)
\end{align}
for all $(u,v) \in K$. \\

\noindent The discretized variant reads as follows:\\
\noindent Find $(\hat u_h, \hat v_h) \in K\cap X_h^p$ such that
\begin{align}\label{vih}
& \langle A'(\varepsilon(\hat u_h)), \varepsilon(u_h-\hat u_h)\rangle + \langle S_h(\hat u_h|_\pom+\hat v_h), (u_h-\hat u_h)|_\pom+v_h-\hat v_h\rangle\nonumber \\ & \qquad \qquad\qquad +j(v_h)-j(\hat v_h)\geq L_h(u_h-\hat u_h,v_h-\hat v_h)
\end{align}
for all $(u_h,v_h) \in K\cap X_h^p$. \\

\begin{thm}\label{a priori}
a) The following a priori estimate holds with $q = \max\{p,2\}$:
\begin{align*}
& \|\hat u-\hat u_h, \hat v-\hat v_h\|_{X^p}^q\\
& \lesssim \inf_{(u_h, v_h) \in K\cap X_h^p} \big\{\|\varepsilon(\hat u-u_h)\|_{L^p(\om)} + \|(\hat u - u_h)|_\pom + \hat v - v_h\|_{W^{\frac{1}{2},2}(\pom)} \\
& \quad + \|\hat v- v_h\|_{L^1(\gs)} \big\}
+\mathrm{dist}_{\Wmg}(V^{-1}(1-K)(\hat u + \hat v - u_0),\Whmg)^2 \ .
\end{align*}
b) If $\hat v \in \widetilde{W}^{\frac{1}{2},2}(\Gamma_s)^n$, e.g.~for $p\geq 2$ or $\Gamma_s = \emptyset$, the estimate can be improved to
\begin{align*}
& \|\hat u-\hat u_h, \hat v-\hat v_h\|_{X^p}^q\\
& \lesssim \inf_{(u_h, v_h) \in K\cap X_h^p} \big\{\|\varepsilon(\hat u-u_h)\|^{\beta}_{L^p(\om)} + \|(\hat u - u_h)|_\pom + \hat v - v_h\|_{W^{\frac{1}{2},2}(\pom)}^2 \\
& \quad + \|\hat v- v_h\|_{L^1(\gs)} \big\}
+\mathrm{dist}_{\Wmg}(V^{-1}(1-K)(\hat u + \hat v - u_0),\Whmg)^2 \ .
\end{align*}
Here $\beta = \frac{2}{3-p}$ for $p<2$ resp.~$\beta = p' = \frac{p}{p-1}$ for $p\geq 2$.
\end{thm}
\begin{proof}
Adding the continuous and discrete variational inequalities, we see that
\begin{align*}
0 & \leq \langle A'(\varepsilon(\hat u)), \varepsilon(\hat u_h) - \varepsilon(\hat u)\rangle + \langle S(\hat u|_\pom+\hat v), (\hat u_h-\hat u)|_\pom+\hat v_h-\hat v\rangle\\
& \quad + j(\hat v_h) - j(\hat v) - L(\hat u_h - \hat u, \hat v_h - \hat v)\\
& \quad + \langle A'(\varepsilon(\hat u_h)), \varepsilon(u_h) - \varepsilon(\hat u_h)\rangle + \langle S_h(\hat u_h|_\pom+\hat v_h), (u_h - \hat u_h)|_\pom+v_h-\hat v_h\rangle\\
& \quad + j(v_h) - j(\hat v_h) - L_h(u_h - \hat u_h, v_h - \hat v_h) \ .
\end{align*}
Hence,
\begin{align*}
&\langle A'(\varepsilon(\hat u)) - A'(\varepsilon(\hat u_h)), \varepsilon(\hat u) - \varepsilon(\hat u_h)\rangle + \langle S((\hat u-\hat u_h)|_\pom+\hat v-\hat v_h), (\hat u-\hat u_h)|_\pom+\hat v-\hat v_h\rangle\\
&\leq
\langle A'(\varepsilon(\hat u)) - A'(\varepsilon(\hat u_h)), \varepsilon(\hat u) - \varepsilon(\hat u_h)\rangle + \langle S((\hat u-\hat u_h)|_\pom+\hat v-\hat v_h), (\hat u-\hat u_h)|_\pom+\hat v-\hat v_h\rangle\\
&\quad+\langle A'(\varepsilon(\hat u)), \varepsilon(\hat u_h) - \varepsilon(\hat u)\rangle + \langle S(\hat u|_\pom+\hat v), (\hat u_h-\hat u)|_\pom+\hat v_h-\hat v\rangle\\
& \quad + j(\hat v_h) - j(\hat v) - L(\hat u_h - \hat u, \hat v_h - \hat v)\\
& \quad + \langle A'(\varepsilon(\hat u_h)), \varepsilon(u_h) - \varepsilon(\hat u_h)\rangle + \langle S_h(\hat u_h|_\pom+\hat v_h), (u_h - \hat u_h)|_\pom+v_h-\hat v_h\rangle\\
& \quad + j(v_h) - j(\hat v_h) - L_h(u_h - \hat u_h, v_h - \hat v_h)\\
& = \langle A'(\varepsilon(\hat u_h)), \varepsilon(u_h) - \varepsilon(\hat u)\rangle + \langle S(\hat u_h|_\pom + \hat v_h), (u_h - \hat u)|_\pom+v_h-\hat v\rangle\\
& \quad + j(v_h)- j(\hat v) - L(u_h-\hat u, v_h - \hat v) - (L_h-L)(u_h-\hat u_h, v_h - \hat v_h)\\
&\quad +\langle (S_h-S)(\hat u_h|_\pom + \hat v_h),(u_h-\hat u_h)|_\pom + v_h-\hat v_h \rangle \\
& = \langle A'(\varepsilon(\hat u)) - A'(\varepsilon(\hat u_h)), \varepsilon(\hat u) - \varepsilon(u_h)\rangle + \langle S((\hat u - \hat u_h)|_\pom + \hat v - \hat v_h), (\hat u - u_h)|_\pom + \hat v - v_h\rangle\\
& \quad +\langle A'(\varepsilon(\hat u)), \varepsilon(u_h)-\varepsilon(\hat u)\rangle + \langle S(\hat u|_\pom + \hat v), (u_h-\hat u)|_\pom + v_h-\hat v \rangle - L(u_h-\hat u, v_h - \hat v) \\
& \quad + j(v_h)- j(\hat v) \\
& \quad - (L_h-L)(u_h-\hat u_h, v_h - \hat v_h) +\langle (S_h-S)(\hat u_h|_\pom + \hat v_h),(u_h-\hat u_h)|_\pom + v_h-\hat v_h \rangle \ .
\end{align*}
Let $p<2$. To bound $\langle A'(\varepsilon(\hat u)) - A'(\varepsilon(\hat u_h)), \varepsilon(\hat u) - \varepsilon(u_h)\rangle$, we use the estimate \eqref{upperbound1} and Young's inequality for any $\delta>0$:
\begin{align*}
\langle A'(\varepsilon(\hat u)) - A'(\varepsilon(\hat u_h)), \varepsilon(\hat u-u_h)\rangle &\lesssim \|\varepsilon(\hat u-\hat u_h)\|_{L^p(\Omega)}^{p-1}\|\varepsilon(\hat u-u_h)\|_{L^p(\Omega)}\\
&\lesssim\delta^{\frac{2}{p-1}} \|\varepsilon(\hat u-\hat u_h)\|_{L^p(\om)}^{2} + \delta^{-\frac{2}{3-p}}\|\varepsilon(\hat u-u_h)\|_{L^p(\om)}^{\frac{2}{3-p}}\ .
\end{align*}
On the other hand, for $p\geq 2$ the upper bound \eqref{upperbound2} yields
\begin{align*}
\langle A'(\varepsilon(\hat u))-A'(\varepsilon(\hat u_h)), \varepsilon(\hat u - u_h)\rangle &\lesssim \|\varepsilon(\hat u - \hat u_h)\|_{L^p(\om)} \|\varepsilon(\hat u - u_h)\|_{L^p(\om)} \\
&\lesssim \delta^p \|\varepsilon(\hat u - \hat u_h)\|_{L^p(\om)}^p + \delta^{-p'} \|\varepsilon(\hat u - \hat u_h)\|_{L^p(\om)}^{p'}\ .
\end{align*}
As for the second term, we use the boundedness of $S$ from $W^{\frac{1}{2},2}(\pom)^n$ to $W^{-\frac{1}{2},2}(\pom)^n$ to estimate
\begin{align*}
& \langle S((\hat u - \hat u_h)|_\pom + \hat v - \hat v_h), (\hat u - u_h)|_\pom + \hat v - v_h\rangle\\
& \lesssim \|(\hat u - \hat u_h)|_\pom + \hat v - \hat v_h\|_{W^{\frac{1}{2},2}(\pom)}\|(\hat u - u_h)|_\pom + \hat v - v_h\|_{W^{\frac{1}{2},2}(\pom)}\\
& \lesssim \delta\|(\hat u - \hat u_h)|_\pom + \hat v - \hat v_h\|_{W^{\frac{1}{2},2}(\pom)}^2+ \delta^{-1}\|(\hat u - u_h)|_\pom + \hat v - v_h\|_{W^{\frac{1}{2},2}(\pom)}^{2}\ .
\end{align*}
Without further assumptions on $\hat v$, we estimate the second line using Cauchy--Schwarz by a multiple of
$$\|\varepsilon(u_h-\hat u)\|_{L^p(\Omega)} + \|(u_h-\hat u)|_\pom + v_h-\hat v \rangle\|_{W^{\frac{1}{2},2}(\pom)}$$
For part b), where $\hat v \in \widetilde{W}^{\frac{1}{2},2}(\Gamma_s)$, one may use the variational inequality for an improved estimate: Substituting $(u,v) = (u_h, \hat v)$ and $(u,v) = (2 \hat u - u_h, \hat v)$ into the variational inequality on $X^p$, we obtain
$$\langle A'(\varepsilon(\hat u)), \varepsilon(u_h) - \varepsilon(\hat u)\rangle + \langle S(\hat u|_\pom+\hat v), (u_h-\hat u)|_\pom\rangle = L(u_h - \hat u,0)\ .$$
With this, the second line reduces to
$\langle S(\hat u|_\pom + \hat v), v_h-\hat v \rangle+L(0, \hat v - v_h)$, i.e.~to
$$- \langle t_0 - S(\hat u|_\pom + \hat v-u_0), v_h - \hat v\rangle = - \langle A'(\varepsilon(\hat u)) \cdot \nu, v_h - \hat v\rangle \leq \|\mathcal{F}\|_{L^{\infty}(\gs)}\|v_{n,h} - \hat v_n\|_{L^1(\gs)}.$$
For the third line,
$$j(v_h) - j(\hat v) = \int_\gs \mathcal{F}(|v_{t,h}| - |\hat v_{t}|) \leq \int_\gs \mathcal{F}(|v_{t,h} - \hat v_t|) \leq \|\mathcal{F}\|_{L^{\infty}(\gs)}\|v_{t,h} - \hat v_t\|_{L^1(\gs)}\ .$$
Finally, the last line simplifies as follows:
\begin{align*}
&- (L_h-L)(u_h-\hat u_h, v_h - \hat v_h) +\langle (S_h-S)(\hat u_h|_\pom + \hat v_h),(u_h-\hat u_h)|_\pom + v_h-\hat v_h \rangle\\
&= \langle (S_h-S)(\hat u_h|_\pom + \hat v_h-u_0),(u_h-\hat u_h)|_\pom + v_h-\hat v_h \rangle \\
& \lesssim \delta^{-1}\|(S_h-S)(\hat u_h|_\pom + \hat v_h-u_0)\|_{W^{-\frac{1}{2},2}(\pom)}^{2} + \delta \|(u_h-\hat u_h)|_\pom + v_h-\hat v_h\|_{W^{\frac{1}{2},2}(\pom)}^2\\
& \leq \delta^{-1}\|(S_h-S)(\hat u_h|_\pom + \hat v_h-u_0)\|_{W^{-\frac{1}{2},2}(\pom)}^2 \\
&\quad + \delta \|(u_h-\hat u)|_\pom + v_h-\hat v\|_{W^{\frac{1}{2},2}(\pom)}^2 + \delta \|(\hat u-\hat u_h)|_\pom + \hat v-\hat v_h\|_{W^{\frac{1}{2},2}(\pom)}^2\ .
\end{align*}
The term involving $S_h- S$ is known to be bounded by \cite{c} $$\mathrm{dist}_{\Wmg}(V^{-1}(1-K)(\hat u + \hat v - u_0),\Whmg)^2\ .$$
To sum up, for general $\hat v$ we obtain for $\alpha = \frac{p}{p-1}$, $\beta = \frac{2}{3-p}$ ($p<2$) resp.~$\alpha=p$, $\beta = p'$ ($p\geq 2$)
\begin{align*}
&\langle A'(\varepsilon(\hat u)) - A'(\varepsilon(\hat u_h)), \varepsilon(\hat u) - \varepsilon(\hat u_h)\rangle + \langle S((\hat u-\hat u_h)|_\pom+\hat v-\hat v_h), (\hat u-\hat u_h)|_\pom+\hat v-\hat v_h\rangle\\
&\lesssim \delta^{\alpha} \|\varepsilon(\hat u - \hat u_h)\|_{L^p(\om)}^q + \delta\|(\hat u - \hat u_h)|_\pom + \hat v - \hat v_h\|_{W^{\frac{1}{2},2}(\pom)}^2+\delta^{-\beta} \|\varepsilon(\hat u-u_h)\|^{\beta}_{L^p(\om)}\\
& \quad + \|\varepsilon(u_h-\hat u)\|_{L^p(\Omega)} + \|(u_h-\hat u)|_\pom + v_h-\hat v \rangle\|_{W^{\frac{1}{2},2}(\pom)}\\
& \quad +  \delta^{-1}\|(\hat u - u_h)|_\pom + \hat v - v_h\|_{W^{\frac{1}{2},2}(\pom)}^2 + \|v_h - \hat v\|_{L^p(\gs)} + \delta \|(u_h-\hat u)|_\pom + v_h-\hat v\|_{W^{\frac{1}{2},2}(\pom)}^2\\
& \quad + \delta^{-1}\mathrm{dist}_{\Wmg}(V^{-1}(1-K)(\hat u + \hat v - u_0),\Whmg)^2 \ .
\end{align*}
The lowest exponents dominate.\\

When $\hat v \in \widetilde{W}^{\frac{1}{2},2}(\Gamma_s)^n$, the estimates yield:
\begin{align*}
&\langle A'(\varepsilon(\hat u)) - A'(\varepsilon(\hat u_h)), \varepsilon(\hat u) - \varepsilon(\hat u_h)\rangle + \langle S((\hat u-\hat u_h)|_\pom+\hat v-\hat v_h), (\hat u-\hat u_h)|_\pom+\hat v-\hat v_h\rangle\\
&\lesssim \delta^{\alpha} \|\varepsilon(\hat u - \hat u_h)\|_{L^p(\om)}^q + \delta\|(\hat u - \hat u_h)|_\pom + \hat v - \hat v_h\|_{W^{\frac{1}{2},2}(\pom)}^2+\delta^{-\beta} \|\varepsilon(\hat u-u_h)\|^{\beta}_{L^p(\om)}\\
& \quad +  \delta^{-1}\|(\hat u - u_h)|_\pom + \hat v - v_h\|_{W^{\frac{1}{2},2}(\pom)}^2 + \|v_h - \hat v\|_{L^p(\gs)}+ \delta \|(u_h-\hat u)|_\pom + v_h-\hat v\|_{W^{\frac{1}{2},2}(\pom)}^2\\
& \quad + \delta^{-1}\mathrm{dist}_{\Wmg}(V^{-1}(1-K)(\hat u + \hat v - u_0),\Whmg)^2  \ .
\end{align*}

Note that as in Lemma \ref{monotony}, the monotony of $A'$ and coercivity of $S$ allow to bound the left hand side from below by $$\|\varepsilon(\hat u - \hat u_h)\|_{L^p(\om)}^q + \|(\hat u-\hat u_h)|_\pom+\hat v-\hat v_h\|^2_{W^{\frac{1}{2},2}(\pom)}\ .$$

Choosing $\delta>0$ sufficiently small, the claimed estimates follow.
\end{proof}

\begin{rem}
a) Theorem \ref{a priori} proves convergence of the proposed FE--BE coupling procedure for quasi--uniform grid refinements. However, generic weak solutions to the contact problem \eqref{CTP} only belong to $X^p$ and not to any higher-order Sobolev space. Therefore the convergence can be arbitrarily slow as the grid size $h$ tends to $0$.\\
b) Like for the $p$--Laplacian operators in \cite{scalar}, under additional assumptions on $A'$ slightly sharper estimates can be obtained with respect to certain quasinorms on $X^p$.
\end{rem}

\section{An a posteriori estimate}\label{apostsection}

%Assume $\hat v \in \widetilde{W}^{\frac{1}{2},2}(\Gamma_s)$. If we consider the variational inequality \eqref{vi} for $v=\hat v$ and with $u \mapsto u$ resp.~$u\mapsto 2\hat u - u$, the problem \eqref{vi} splits into an interior equation and an inequality on the boundary:\\
%
%Find $(\hat u, \hat v) \in K$ such that for all $(u,v)\in K$:
%\begin{align*}
%&\langle A'(\varepsilon(\hat u)), \varepsilon(u)\rangle + \langle S(\hat u|_\pom + \hat v), u|_\pom\rangle = \int_\om f u +\langle t_0+Su_0,u \rangle = L(u,0)\ ,\\
%&\langle S(\hat u|_\pom + \hat v),v-\hat v \rangle + j(v)-j(\hat v)\geq \langle t_0 + S u_0, v-\hat v  \rangle = L(0,v-\hat v)\ .
%\end{align*}
%We obtain a variant of Galerkin orthogonality in the interior:
%\begin{align*}
%&\langle A'(\varepsilon(\hat u))- A'(\varepsilon(\hat u_h)), \varepsilon(u_h)\rangle + \langle S((\hat u-\hat u_h)|_\pom + \hat v-\hat v_h), u_h|_\pom\rangle \\
%& \qquad \qquad + \langle (S-S_h)(\hat u_h|_\pom +\hat v_h-u_0), u_h|_\pom\rangle= 0\ .
%\end{align*}

If we consider the variational inequality \eqref{vih} for $v_h=\hat v_h$ and with $u_h \mapsto u_h$ resp.~$u_h\mapsto 2\hat u_h - u_h$, Problem \eqref{vih} splits into an interior equation and an inequality on the boundary: For all $(u_h,v_h)\in K\cap X^p_h$:
\begin{align}\label{visplith}
&\langle A'(\varepsilon(\hat u_h)), \varepsilon(u_h)\rangle + \langle S_h(\hat u_h|_\pom + \hat v_h), u_h|_\pom\rangle = \int_\om f u_h +\langle t_0+S_hu_0,u_h \rangle = L_h(u_h,0)\ ,\nonumber\\
&\langle S_h(\hat u_h|_\pom + \hat v_h),v_h-\hat v_h \rangle + j(v_h)-j(\hat v_h)\geq \langle t_0 + S u_0, v_h-\hat v_h  \rangle = L_h(0,v_h-\hat v_h)\ .
\end{align}
For the continuous inequality, we only get a weaker assertion because $u|_\pom + v$ needs to be in $W^{\frac{1}{2},2}(\pom)$. Choosing $u=\hat u +\hat u_h-u_h$, $v= \hat v + \hat v_h - v_h$ for any $(u_h, v_h) \in X^p_h$ with $v_h \leq \hat v + \hat v_h$ transforms \eqref{vi} into the estimate
\begin{align}\label{visplit}
&\langle A'(\varepsilon(\hat u)), \varepsilon(u_h-\hat u_h)\rangle + \langle S(\hat u|_\pom + \hat v), (u_h-\hat u_h)|_\pom + v_h - \hat v_h\rangle \nonumber \\ & \qquad \leq  j(\hat v+\hat v_h-v_h) - j(\hat v)+ L(u_h-\hat u_h, v_h - \hat v_h)\ .
\end{align}
In combination with the coercivity estimates, we may start to derive an a posteriori estimate:
\begin{align*}
& \|\varepsilon(\hat u-\hat u_h)\|^q_{L^p(\om)} +\|(\hat u - \hat u_h)|_\pom + \hat v - \hat v_h\|_{W^{\frac{1}{2},2}(\pom)}^2\\
& \lesssim \langle A'(\varepsilon(\hat u))-A'(\varepsilon(\hat u_h)), \varepsilon(\hat u - u_h)\rangle + \langle A'(\varepsilon(\hat u))-A'(\varepsilon(\hat u_h)), \varepsilon(u_h - \hat u_h)\rangle\\
& \quad + \langle S((\hat u-\hat u_h)|_\pom + \hat v-\hat v_h), (\hat u - u_h)|_\pom + \hat v - v_h\rangle\\
& \quad + \langle S((\hat u-\hat u_h)|_\pom + \hat v-\hat v_h), (u_h - \hat u_h)|_\pom + v_h - \hat v_h\rangle
\end{align*}
We consider the second and fourth term on the right hand side,
\begin{align*}
&\langle A'(\varepsilon(\hat u)), \varepsilon(u_h - \hat u_h)\rangle- \langle A'(\varepsilon(\hat u_h)), \varepsilon(u_h - \hat u_h)\rangle\\
& \quad + \langle S(\hat u|_\pom+ \hat v), (u_h - \hat u_h)|_\pom + v_h - \hat v_h\rangle -\langle S(\hat u_h|_\pom+ \hat v_h), (u_h - \hat u_h)|_\pom + v_h - \hat v_h\rangle\ .
\end{align*}
Applying the equality in \eqref{visplith} to $$\langle A'(\varepsilon(\hat u_h)), \varepsilon(u_h - \hat u_h)\rangle + \langle S(\hat u_h|_\pom+ \hat v_h), (u_h - \hat u_h)|_\pom + v_h - \hat v_h\rangle\ ,$$ and inequality \eqref{visplit} to
$$\langle A'(\varepsilon(\hat u)), \varepsilon(u_h - \hat u_h)\rangle + \langle S(\hat u|_\pom+ \hat v), (u_h - \hat u_h)|_\pom + v_h - \hat v_h\rangle\ ,$$
 we estimate their sum by
\begin{align*}
&-L_h(u_h-\hat u_h,0)+ j(\hat v+\hat v_h-v_h)-j(\hat v)  + L(u_h-\hat u_h, v_h - \hat v_h)\\
&\qquad -\langle S_h(\hat u_h|_\pom + \hat v_h), v_h - \hat v_h \rangle+\langle(S_h-S)(\hat u_h|_\pom + \hat v_h),(u_h - \hat u_h)|_\pom + v_h - \hat v_h \rangle\ .
\end{align*}
For $$\langle A'(\varepsilon(\hat u)), \varepsilon(\hat u - u_h)\rangle+ \langle S(\hat u|_\pom + \hat v), (\hat u - u_h)|_\pom + \hat v - v_h\rangle\ ,$$ we use the variational inequality \eqref{vi} with $(u,v) = (u_h, v_h)$ to conclude
\begin{align*}
& \|\varepsilon(\hat u-\hat u_h)\|^q_{L^p(\om)} +\|(\hat u - \hat u_h)|_\pom + \hat v - \hat v_h\|_{W^{\frac{1}{2},2}(\pom)}^2\\
& \lesssim L(\hat u - u_h, \hat v - v_h) + j(v_h)-j(\hat v)-\langle A'(\varepsilon(\hat u_h)), \varepsilon(\hat u - u_h)\rangle \\
& \quad - \langle S(\hat u_h|_\pom +\hat v_h), (\hat u - u_h)|_\pom + \hat v - v_h\rangle-\langle S_h(\hat u_h|_\pom + \hat v_h), v_h - \hat v_h \rangle\\
& \quad -L_h(u_h-\hat u_h,0)+ j(\hat v+\hat v_h-v_h)- j(\hat v) + L(u_h-\hat u_h, v_h - \hat v_h)\\
& \quad +\langle(S_h-S)(\hat u_h|_\pom + \hat v_h),(u_h - \hat u_h)|_\pom + v_h - \hat v_h \rangle \\
& = \int_\om f(\hat u - u_h) + \langle t_0 + S u_0 ,(\hat u - u_h)|_\pom+ \hat v - v_h) + j(\hat v+\hat v_h-v_h)+ j(v_h)-2j(\hat v)\\
& \quad -\langle A'(\varepsilon(\hat u_h)), \varepsilon(\hat u - u_h)\rangle - \langle S_h(\hat u_h|_\pom +\hat v_h), (\hat u - u_h)|_\pom + \hat v - v_h\rangle\\
& \quad -\langle S_h(\hat u_h|_\pom + \hat v_h), v_h - \hat v_h \rangle - \langle (S_h-S)u_0,(u_h-\hat u_h)|_\pom\rangle + \langle t_0 + S u_0,v_h - \hat v_h\rangle\\
& \quad +\langle(S_h-S)(\hat u_h|_\pom + \hat v_h),(\hat u - \hat u_h)|_\pom + \hat v - \hat v_h \rangle \\
& = \int_\om f(\hat u - u_h)  + j(\hat v+\hat v_h-v_h)+ j(v_h)-2j(\hat v)-\langle A'(\varepsilon(\hat u_h)), \varepsilon(\hat u - u_h)\rangle \\
& \quad + \langle t_0- S_h(\hat u_h|_\pom +\hat v_h-u_0), (\hat u - u_h)|_\pom + \hat v - \hat v_h\rangle\\
& \quad  +\langle(S_h-S)(\hat u_h|_\pom + \hat v_h-u_0),(\hat u - \hat u_h)|_\pom + \hat v - \hat v_h \rangle \ .
\end{align*}
%& = L(\hat u-u_h, 0)-\langle A'(\varepsilon(\hat u_h)), \varepsilon(\hat u - u_h)\rangle\\
%&\quad -\langle S(\hat u_h|_\pom + \hat v_h, (\hat u - u_h)|_\pom + \hat v - v_h\rangle \\
%& \quad + \langle S (\hat u|_\pom + \hat v),\hat v - v_h\rangle + \langle S((\hat u-\hat u_h)|_\pom + \hat v-\hat v_h), v_h - \hat v_h\rangle\\
%& \quad + \langle (S-S_h)(\hat u_h|_\pom +\hat v_h-u_0), (u_h - \hat u_h)|_\pom\rangle\\
%& = L(\hat u-u_h, 0)-\langle A'(\varepsilon(\hat u_h)), \varepsilon(\hat u - u_h)\rangle - \langle S(\hat u_h|_\pom+\hat v_h), (\hat u - u_h)|_\pom\rangle\\
%& \quad +  \langle S((\hat u-\hat u_h)|_\pom + \hat v-\hat v_h), \hat v - \hat v_h\rangle\\
%& \quad + \langle (S-S_h)(\hat u_h|_\pom +\hat v_h-u_0), (u_h - \hat u_h)|_\pom\rangle\\
%& = \int_\om f (\hat u - u_h) - \sum_{E \subset \om} \int_E[A'(\varepsilon(\hat u_h))\nu](\hat u - u_h)|_E\\
%& \quad +\langle t_0 + S(u_0-\hat u_h|_\pom - \hat v_h)-A'(\varepsilon(\hat u_h))\nu, (\hat u - u_h)|_\pom\rangle\\
%& \quad +  \langle S((\hat u-\hat u_h)|_\pom + \hat v-\hat v_h), \hat v - \hat v_h\rangle\\
%& \quad + \langle (S-S_h)(\hat u_h|_\pom +\hat v_h-u_0), (u_h - \hat u_h)|_\pom\rangle\ .
%\end{align*}
The first term is estimated as usual for $u_h = \hat u_h + \Pi_h(\hat u - \hat u_h)$ using the H\"{o}lder inequality and the properties of a Clement interpolation operator $\Pi_h$ (see e.g.~\cite{bs}):
$$\int_\om f (\hat u - u_h) \lesssim \|\hat u - \hat u_h\|_{W^{1,p}(\om)}\ \left(\sum_{T\subset \om} h_T^{p'}\|f\|_{L^{p'}(T)}^{p'}\right)^{1/{p'}}\qquad (p'=\textstyle{\frac{p}{p-1}})$$
Similarly, integrating by parts we obtain
$$\langle A'(\varepsilon(\hat u_h)), \varepsilon(\hat u - u_h)\rangle  = \sum_{E \subset \om} \int_E[A'(\varepsilon(\hat u_h))\nu](\hat u - u_h)|_\pom +\langle A'(\varepsilon(\hat u_h))\nu, (\hat u - u_h)|_\pom\rangle_\pom$$
with
$$\sum_{E \subset \om} \int_E[A'(\varepsilon(\hat u_h))\nu](\hat u - u_h)|_\pom \lesssim  \|\hat u - \hat u_h\|_{W^{1,p}(\om)}\ \left(\sum_{E \subset \om} h_E \|[A'(\varepsilon(\hat u_h))\nu]\|_{L^{p'}(E)}^{p'}\right)^{1/{p'}}\ .$$
It remains to consider the boundary contributions. To do so, recall the strong formulation of the contact conditions in terms of $\sigma_n(u)$ and $\sigma_t(u)$ on $\gs$,
\begin{align*}& \sigma_n(u)\leq 0\ , \ v_n \leq 0\ , \ \sigma_n(u)v_n=0\ , \\
&|\sigma_t(u)|\leq \mathcal{F} \ , \ \sigma_t(u) v_t+\mathcal{F}|v_t|=0\ .
\end{align*}
Then, substituting $v_h = \hat v_h$, we obtain
$$j(\hat v + \hat v_h - v_h) = j(\hat v) = \int_\gs \mathcal{F}|\hat v_t| = -\langle \sigma_t(\hat u), \hat v_t\rangle = -\langle\sigma(\hat u), \hat v\rangle\ .$$
Also, $$j(\hat v_h) - \langle A'(\varepsilon(\hat u_h))\nu,\hat v_h\rangle \leq \int_\gs \left\{\mathcal{F}|\hat v_{h,t}|+ \sigma_t(\hat u_h)\hat v_{h,t}\right\} + \int_\gs (\sigma_n(\hat u_h) \hat v_{n,h})_+\ .$$
Together, the terms 
\begin{align*}
&j(\hat v+\hat v_h-v_h)+j(v_h)- 2j(\hat v)-\langle A'(\varepsilon(\hat u_h))\nu, (\hat u - u_h)|_\pom\rangle_\pom\\
&\quad  = - j(\hat v) +j(\hat v_h) -\langle A'(\varepsilon(\hat u_h))\nu, \hat v_h\rangle_\pom-\langle A'(\varepsilon(\hat u_h))\nu, (\hat u - u_h)|_\pom +\hat v- \hat v_h - \hat v\rangle_\pom
\end{align*}
are hence dominated by
\begin{align*}
&\langle\sigma(\hat u), \hat v\rangle + \int_\gs \left\{\mathcal{F}|\hat v_{h,t}|+ \sigma_t(\hat u_h)\hat v_{h,t}\right\} + \int_\gs (\sigma_n(\hat u_h) \hat v_{n,h})_+\\
& \quad -\langle A'(\varepsilon(\hat u_h))\nu, (\hat u - u_h)|_\pom +\hat v- \hat v_h - \hat v\rangle_\pom\\
&=\int_\gs \left\{\mathcal{F}|\hat v_{h,t}|+ \sigma_t(\hat u_h)\hat v_{h,t}\right\} + \int_\gs (\sigma_n(\hat u_h) \hat v_{n,h})_+  \\
& \quad -\langle A'(\varepsilon(\hat u_h))\nu, (\hat u - u_h)|_\pom+\hat v- \hat v_h \rangle_\pom+ \langle\sigma(\hat u) - \sigma(\hat u_h), \hat v\rangle\ .
\end{align*}
We split the $\sigma$--term into tangential and normal parts
$$ \langle \sigma(\hat u) -  \sigma(\hat u_h), \hat v\rangle= \langle \sigma_n(\hat u) - \sigma_n(\hat u_h), \hat v_n \rangle + \langle \sigma_t(\hat u) - \sigma_t(\hat u_h), \hat v_t \rangle\ ,$$
and estimate the normal part as follows ($r' = \frac{r}{r-1}$):
$$\langle \sigma_n(\hat u) - \sigma_n(\hat u_h), \hat v_n \rangle \leq -\langle \sigma_n(\hat u_h)_+, \hat v_n \rangle \lesssim \|\sigma_n(\hat u_h)_+\|_{\widetilde{W}^{1-\frac{1}{r},r'}(\gs)}\ . $$
For the tangential contribution, involving the Tresca friction, we find it convenient to write $\sigma_t(\hat u) = -\zeta \mathcal{F}$ with $|\zeta|\leq 1$ and $|v_t| = \zeta v_t$. Then
\begin{align*}
&\langle \sigma_t(\hat u) - \sigma_t(\hat u_h), \hat v_t\rangle = -\langle \zeta \mathcal{F}, \hat v_t\rangle - \langle\sigma_t(\hat u_h), \hat v_t\rangle = -\langle \mathcal{F}, |\hat v_t|\rangle - \langle\sigma_t(\hat u_h),\hat v_t\rangle \\
& \leq \langle (|\sigma_t(\hat u_h)|-\mathcal{F})_+,|\hat v_t|\rangle \lesssim   \|(|\sigma_t(\hat u_h)|-\mathcal{F})_+\|_{\widetilde{W}^{-1+\frac{1}{r},r'}(\gs)}\ .
\end{align*}
We conclude
\begin{align*}
& \|\varepsilon(\hat u-\hat u_h)\|^q_{L^p(\om)} +\|(\hat u - \hat u_h)|_\pom + \hat v - \hat v_h\|_{W^{\frac{1}{2},2}(\pom)}^q\\
& \lesssim\|\varepsilon(\hat u-\hat u_h)\|^q_{L^p(\om)} +\|(\hat u - \hat u_h)|_\pom + \hat v - \hat v_h\|_{W^{\frac{1}{2},2}(\pom)}^2\\
& \lesssim \|\hat u - \hat u_h\|_{W^{1,p}(\om)}\ \left(\sum_{T\subset \om} h_T^{p'}\|f\|_{L^{p'}(T)}^{p'}\right)^{1/{p'}}  \\
& \quad +\|\hat u - \hat u_h\|_{W^{1,p}(\om)}\ \left(\sum_{E \subset \om} h_E \|[A'(\varepsilon(\hat u_h))\nu]\|_{L^{p'}(E)}^{p'}\right)^{1/{p'}}\\
&\quad +\int_\gs \left\{\mathcal{F}|\hat v_{h,t}|+ \sigma_t(\hat u_h)\hat v_{h,t}\right\} + \int_\gs (\sigma_n(\hat u_h) \hat v_{n,h})_+  \\
& \quad +\| t_0+ S_h(u_0-\hat u_h|_\pom +\hat v_h)-A'(\varepsilon(\hat u_h))\nu\|_{{W}^{1-\frac{1}{r},r'}(\pom)}^{q'} \\
&  \quad + \|\sigma_n(\hat u_h)_+\|_{\widetilde{W}^{1-\frac{1}{r},r'}(\gs)} +\|(|\sigma_t(\hat u_h)|-\mathcal{F})_+\|_{\widetilde{W}^{-1+\frac{1}{r},r'}(\gs)}\\
& \quad  +\langle(S_h-S)(\hat u_h|_\pom + \hat v_h-u_0),(\hat u - \hat u_h)|_\pom + \hat v - \hat v_h \rangle\ .
\end{align*}
Summing up:
\begin{thm}\label{a posteriori}
Let $r=\min\{p,2\}$ and $q=\max\{p,2\}$. The following a posteriori estimate holds:
\begin{align*}
& \|\hat u-\hat u_h, \hat v - \hat v_h\|^q_{X^p}\\
& \lesssim \left(\sum_{T\subset \om} h_T^{p'}\|f\|_{L^{p'}(T)}^{p'}\right)^{q'/{p'}}  +\quad \left(\sum_{E \subset \om} h_E \|[A'(\varepsilon(\hat u_h))\nu]\|_{L^{p'}(E)}^{p'}\right)^{{q'}/{p'}}\\
& \quad +\| t_0+ S_h(u_0-\hat u_h|_\pom +\hat v_h)-A'(\varepsilon(\hat u_h))\nu\|_{{W}^{1-\frac{1}{r},r'}(\pom)}^{q'} \\
&\quad +\int_\gs \left\{\mathcal{F}|\hat v_{h,t}|+ \sigma_t(\hat u_h)\hat v_{h,t}\right\} + \int_\gs (\sigma_n(\hat u_h) \hat v_{n,h})_+  \\
&  \quad + \|\sigma_n(\hat u_h)_+\|_{\widetilde{W}^{1-\frac{1}{r},r'}(\gs)} +\|(|\sigma_t(\hat u_h)|-\mathcal{F})_+\|_{\widetilde{W}^{-1+\frac{1}{r},r'}(\gs)}\\
& \quad  +\|(S_h-S)(\hat u_h|_\pom + \hat v_h-u_0)\|_{\widetilde{W}^{1-\frac{1}{r},r'}(\gs)}^2\ .
\end{align*}

\begin{rem}
Adapting the interpolation operator $\Pi_h$ to include $\hat v - \hat v_h$ on $\gs$, it might be possible to improve the term $\| t_0+ S_h(u_0-\hat u_h|_\pom +\hat v_h)-A'(\varepsilon(\hat u_h))\nu\|_{\widetilde{W}^{1-\frac{1}{r},r'}(\gs)}^{q'}$ to $\| t_0+ S_h(u_0-\hat u_h|_\pom +\hat v_h)-A'(\varepsilon(\hat u_h))\nu\|_{\widetilde{W}^{\frac{1}{2},2}(\gs)}^{2}$.\end{rem}
%\begin{align*}
%& \|\hat u-\hat u_h, \hat v-\hat v_h\|_X^2 + \|(\hat u - \hat u_h)|_\pom + \hat v - \hat v_h\|_{W^{\frac{1}{2},2}(\pom)}^2\\
%& \lesssim \|\varepsilon(\hat u-\hat u_h)\|^2_{L^p(\om)} +\|(\hat u - \hat u_h)|_\pom + \hat v - \hat v_h\|_{W^{\frac{1}{2},2}(\pom)}^2\\
%&\lesssim \left(\sum_{K\subset \om} h_K^{p'}\|f\|_{L^{p'}(K)}^{p'}\right)^{2/{p'}}+\left(\sum_{E \subset \om} h_E \|[A'(\varepsilon(\hat u_h))\nu]\|_{L^{p'}(E)}^{p'}\right)^{2/{p'}}\\
%& \qquad +\|t_0 + S(u_0-\hat u_h|_\pom - \hat v_h)-A'(\varepsilon(\hat u_h))\nu\|_{W^{-1+\frac{1}{p},p'}(\pom)}^2\\
%& \qquad + \|\sigma_n(\hat u_h)_+\|^2_{W^{-1+\frac{1}{p},p'}(\gs)} +  \int_\gs \sigma_n(\hat u_h)_- \hat v_{n,h}\\
%& \qquad + \|(|\sigma_t(\hat u_h)|-\mathcal{F})_+\|^2_{\widetilde{W}^{-1+\frac{1}{p},p'}(\gs)}\\
%&\qquad + \int_\gs |(|\sigma_t(\hat u_h)|-\mathcal{F})_-| |\hat v_{t,h}| + \int_\gs (\sigma_t(\hat u_h)\hat v_{t,h})_+\\
%& \qquad + \|(S-S_h)(\hat u_h|_\pom +\hat v_h-u_0)\|_{W^{-1+\frac{1}{p},p'}(\pom)}^2\ .
%\end{align*}
\end{thm}

\section{Formulation in terms of layer potentials}\label{fulldisc}

In practice, one would like to estimate the numerical error without a priori information about $S-S_h$. This is achieved by formulating the problem directly in terms of the layer potentials $\mathcal{V},\mathcal{W},\mathcal{K},\mathcal{K}'$ rather than $S=\mathcal{W}+(1-\mathcal{K}')\mathcal{V}^{-1}(1-\mathcal{K})$. The arguments are a notationally more involved variant of those in Section \ref{apostsection}, and we only outline them.

We consider the space $$Y^p = X^p \times W^{-\frac{1}{2},2}(\pom)^n\ ,$$
equipped with the norm $$\|u,v,\phi\|_{Y^p} = \|u\|_{W^{1,p}(\om)} + \|v\|_{\widetilde{W}^{1-\frac{1}{p},p}(\gs)} + \|u|_\pom + v\|_{W^{\frac{1}{2},2}(\pom)}+\|\phi\|_{W^{-\frac{1}{2},2}(\pom)} \ .$$ From Lemma \ref{normequiv} we conclude that $(Y^p, \|\cdot\|_{Y^p})$ is a Banach space and
$$|u,v,\phi|_{Y^p} = \|u\|_{W^{1,p}(\om)} +  \|u|_\pom + v\|_{W^{\frac{1}{2},2}(\pom)}+\|\phi\|_{W^{-\frac{1}{2},2}(\pom)}$$
an equivalent norm on $Y^p$. We consider the discretization in finite dimensional subspaces $Y_h^p = X_h^p\times \Whmg^n$ of $Y^p$.

In order to show coercivity, we use a theoretical stabilization as in \cite{dirk}: Let $r_1, \dots, r_D$ a basis of the space of rigid body motions, and consider their orthogonal projections $\xi_1,\dots, \xi_D$ onto $L^2(\pom)$. The arguments in \cite{dirk}, Lemma 4 and Proposition 5, show that $|u,v,\phi|_{Y^p}$ is equivalent to the norm
\begin{align}\label{stabnorm}
|u,v,\phi|_{Y^p,s}^2 &= \|\varepsilon(u)\|^2_{L^{p}(\om)} +  \langle\mathcal{W}(u|_\pom + v),u|_\pom + v\rangle+\langle \phi,\mathcal{V}\phi\rangle\\
& \qquad + \sum_{j=1}^D|\langle \xi_j, (1-\mathcal{K})(u|_\pom + v)+\mathcal{V}\phi\rangle|^2.
\end{align}

On $Y^p$, we have the following equivalent formulation of the contact problem \eqref{CTP}:
Find $(\hat u, \hat v, \hat \phi) \in K'= (K\cap X^p)\times W^{-\frac{1}{2},2}(\pom)^n$ such that for all $(u,v,\phi)\in K'$:
%\begin{align}\label{layerpotVI}
%\langle A'(\varepsilon(\hat u)), \varepsilon(u)\rangle + \langle \mathcal{W}(\hat u|_\pom + \hat v)+(\mathcal{K}'-1)\hat \phi, u|_\pom\rangle& = \int_\om f u +\langle t_0+\mathcal{W}u_0,u \rangle \ ,\nonumber\\
%\langle \mathcal{W}(\hat u|_\pom + \hat v) + (\mathcal{K}'-1)\hat \phi,v-\hat v \rangle + j(v)-j(\hat v)&\geq \langle t_0 + \mathcal{W} u_0, v-\hat v  \rangle \ ,\nonumber\\
%\langle\phi, \mathcal{V}\hat \phi + (1-\mathcal{K})(\hat u|_\pom + \hat v) \rangle &= \langle \phi, (1-\mathcal{K})u_0\rangle\ .
%\end{align}
%The latter can be written as
\begin{align}\label{layerpotVI}
B(\hat u, \hat v, \hat \phi; u-\hat u, v-\hat v, \phi-\hat \phi)+ j(v)-j(\hat v) \geq \Lambda(u-\hat u, v-\hat v, \phi-\hat \phi)
\end{align}
%for all $(u,v,\phi)\in K'$,
with
\begin{eqnarray*}
B(u, v, \phi; \tilde u, \tilde v, \tilde \phi)&=&\langle A'(\varepsilon(u)), \varepsilon(\tilde u)
\rangle +\langle \mathcal{W}(u|_\pom+v) + (\mathcal{K}'-1) \phi,
\tilde u|_\pom +\tilde v\rangle\\
&&\quad+\langle \tilde \phi, \mathcal{V} \phi +
(1-\mathcal{K})(u|_\pom+v)\rangle,\\
\Lambda(u, v, \phi) &=& \langle t_0 + \mathcal{W} u_0, u|_\pom
+v\rangle + \int_\om f u+\langle \phi, (1-\mathcal{K}) u_0\rangle.
\end{eqnarray*}
The discretized problem is obtained by restricting to $Y^p_h$, and we denote its solution by $(\hat u_h, \hat v_h, \hat \phi_h)$.
We also consider a stabilized problem that for all $(u_h, v_h, \phi_h)\in K'\cap Y^p_h$
\begin{align*}
&\widetilde B(\hat u_{s,h}, \hat v_{s,h}, \hat \phi_{s,h}; u_h-\hat u_{s,h}, v_h-\hat v_{s,h}, \phi_h-\hat \phi_{s,h})+ j(v_h)-j(\hat v_{s,h}) \\
& \qquad \geq \widetilde\Lambda(u_h-\hat u_{s,h}, v_h-\hat v_{s,h}, \phi_h-\hat \phi_{s,h})\ ,
\end{align*}
where
\begin{align*}
\widetilde B(u, v, \phi; \tilde u, \tilde v, \tilde \phi) &= B(u, v, \phi; \tilde u, \tilde v, \tilde \phi) \\
& \ \ + \sum_{j=1}^D \langle\xi_j, \mathcal{V}\phi+(1-\mathcal{K})(u|_\pom + v) \rangle \langle\xi_j, \mathcal{V}\tilde\phi+(1-\mathcal{K})(\tilde u|_\pom + \tilde v) \rangle
\end{align*}
respectively
$$\widetilde \Lambda(u, v, \phi) = \Lambda(u, v, \phi) + \sum_{j=1}^D \langle\xi_j, (1-\mathcal{K})u_0 \rangle \langle\xi_j, \mathcal{V}\phi+(1-\mathcal{K})( u|_\pom +  v) \rangle\ .
$$
Because the variational inequality \eqref{layerpotVI} is an equality in $\phi$, as in \cite{dirk}, Proposition 3, the solution to the stabilized and nonstabilized problems coincide, $(\hat u_{h}, \hat v_{h}, \hat \phi_{h})=(\hat u_{s,h}, \hat v_{s,h}, \hat \phi_{s,h})$. However, the stabilized variational inequality is coercive in the stabilized norm \eqref{stabnorm}:
\begin{align*}
& \|\varepsilon(\hat u-\hat u_h)\|^q_{L^p(\om)} +\langle\mathcal{W}((\hat u - \hat u_h)|_\pom + \hat v - \hat v_h),(\hat u - \hat u_h)|_\pom + \hat v - \hat v_h \rangle\\
& \quad + \langle \mathcal{V} (\hat \phi - \hat \phi_h), \hat \phi - \hat \phi_h\rangle  +  \sum_{j=1}^D|\langle \xi_j, (1-\mathcal{K})((\hat u-\hat u_h)|_\pom + \hat v-\hat v_h)+\mathcal{V}(\hat \phi-\hat \phi_h)\rangle|^2\\
& \lesssim \langle A'(\varepsilon(\hat u))-A'(\varepsilon(\hat u_h)), \varepsilon(\hat u - \hat u_h)\rangle \\
& \quad + \langle\mathcal{W}((\hat u - \hat u_h)|_\pom + \hat v - \hat v_h) + (\mathcal{K}'-1)(\hat \phi - \hat \phi_h),(\hat u - \hat u_h)|_\pom + \hat v - \hat v_h \rangle\\
& \quad + \langle (1-\mathcal{K})((\hat u - \hat u_h)|_\pom + \hat v - \hat v_h)+\mathcal{V} (\hat \phi - \hat \phi_h), \hat \phi - \hat \phi_h\rangle\\
& \quad + \sum_{j=1}^D|\langle \xi_j, (1-\mathcal{K})((\hat u-\hat u_h)|_\pom + \hat v-\hat v_h)+\mathcal{V}(\hat \phi-\hat \phi_h)\rangle|^2
\end{align*}
Proceeding as in Section \ref{apostsection}, we obtain:
\begin{thm}\label{a posteriori}
Let $r=\min\{p,2\}$ and $q=\max\{p,2\}$. The following a posteriori estimate holds:
\begin{align*}
& \|\hat u-\hat u_h, \hat v - \hat v_h, \hat \phi - \hat \phi_h\|^q_{Y^p}\\
& \lesssim \left(\sum_{T\subset \om} h_T^{p'}\|f\|_{L^{p'}(T)}^{p'}\right)^{q'/{p'}}  +\quad \left(\sum_{E \subset \om} h_E \|[A'(\varepsilon(\hat u_h))\nu]\|_{L^{p'}(E)}^{p'}\right)^{{q'}/{p'}}\\
& \quad +\|t_0-\mathcal{W}(\hat u_h|_\pom+\hat v_h-u_0) - (\mathcal{K}'-1)
\hat \phi_h-A'(\varepsilon(\hat u_h))\nu\|_{{W}^{1-\frac{1}{r},r'}(\pom)}^{q'} \\
& \quad + \|\mathcal{V}
\hat \phi_h + (1-\mathcal{K})(\hat u_h|_\pom+\hat
v_h-u_0)\|_{W^{-\frac{1}{2},2}(\pom)}^2\\
&\quad +\int_\gs \left\{\mathcal{F}|\hat v_{h,t}|+ \sigma_t(\hat u_h)\hat v_{h,t}\right\} + \int_\gs (\sigma_n(\hat u_h) \hat v_{n,h})_+  \\
&  \quad + \|\sigma_n(\hat u_h)_+\|_{\widetilde{W}^{1-\frac{1}{r},r'}(\gs)} +\|(|\sigma_t(\hat u_h)|-\mathcal{F})_+\|_{\widetilde{W}^{-1+\frac{1}{r},r'}(\gs)}\ .
\end{align*}
\end{thm}

%The key identity for $B$ is
%\begin{lem} \label{Bcoercive}
%For all $(\hat u,\hat v,\hat \phi), (u,v,\phi) \in Y^p = X^p \times W^{-\frac{1}{2},2}(\pom)^2$ such that  $\|\varepsilon(\hat u)\|_{L^p(\om)}, \|\varepsilon(u)\|_{L^p(\om)}<C$ we have
%\begin{align*}
%& \|\varepsilon(\hat u - u)\|_{L^{p}(\om)}^2 + \|(\hat u - u)|_\pom + \hat v - v\|_\Wg^2 + \|\hat \phi- \phi\|_{W^{-\frac{1}{2},2}(\pom)}^2\\
%& \lesssim_C  B(\hat u,\hat v,\hat \phi;\hat u- u, \hat v- v, \eta) -
%B(u,v,\phi;\hat u- u, \hat v- v, \eta),
%\end{align*}
%where $2 \eta = \hat \phi -\phi + V^{-1}(1-K)((\hat u-u)|_\pom +
%\hat v-v)$.
%\end{lem}
%\begin{proof}
%This follows from the identity
%\begin{align*}
%& B(\hat u,\hat v,\hat \phi;\hat u- u, \hat v- v,
%\eta) -
%B(u,v,\phi;\hat u- u, \hat v- v, \eta)\\
%&= \langle A'(\varepsilon(\hat u)) - A'(\varepsilon(u)), \varepsilon(\hat u) - \varepsilon(u)\rangle +
%\textstyle{\frac{1}{2}} \langle \mathcal{W}((\hat u - u)|_\pom +
%\hat v - v),(\hat u - u)|_\pom + \hat v - v)\rangle  \\
%&\hphantom{=} +
%\textstyle{\frac{1}{2}} \langle S((\hat u - u)|_\pom + \hat v -
%v),(\hat u - u)|_\pom + \hat v - v)\rangle+
%\textstyle{\frac{1}{2}}\langle \mathcal{V}(\hat \phi - \phi), \hat
%\phi - \phi \rangle\ .
%\end{align*}
%Also
%$$\|\eta\|_{W^{-\frac{1}{2},2}(\pom)}\lesssim \|\hat \phi - \phi\|_{W^{-\frac{1}{2},2}(\pom)}+\|(\hat u - u)|_\pom + \hat v - v\|_{W^{\frac{1}{2},2}(\pom)}\ .$$
%\end{proof}

\section{Appendix -- An improved error estimate for the scalar $p$--Laplacian}
Consider the following scalar transmission problem for $p\geq 2$:
\begin{eqnarray}\label{TP}
-\mathrm{div}\, A'(\nabla u)&=& f \quad \text{in $\om$,} \nonumber \\
- \Delta u_c &=& 0 \quad \text{in $\omc$,} \nonumber\\
A'(u) \nu -\partial_\nu u_c&=& t_0 \quad \text{on $\pom$,} \label{diff}\\
u-u_c&=& u_0\quad \text{on $\gt$,}\nonumber
\end{eqnarray}
On $\Gamma_s$, contact conditions corresponding to Tresca friction are imposed in terms of the stress $\sigma(u) = -A'(\nabla u) \nu$,
$$
|\sigma(u)|\leq g \ , \ \sigma(u) (u_0+u_c-u)+g|u_0+u_c-u|)=0\ .
$$
A radiation condition holds for $|x| \to \infty$:
$$u(x) = a + o(1)\ ,$$
\noindent and for simplicity of notation we assume $a=0$. Here $A' : L^p(\om)^2 \to L^{p'}(\om)^2 $ is assumed to be a bounded, continuous and uniformly monotone operator, so that in particular
\begin{align*}
\langle A'(x)-A'(y), x-y \rangle &\gtrsim  \|x-y\|_{L^p(\om)}^p \ ,\\
\langle A'(x)-A'(y), z \rangle &\lesssim (\|x\|_{L^p(\om)}+ \|y\|_{L^p(\om)})^{p-2} \|x-y\|_{L^p(\om)} \|z\|_{L^p(\om)}\ ,
\end{align*}
The data belong to the following spaces:
$$f\in L^{p'}(\Omega),\,\, u_0 \in W^{\frac{1}{2},2}(\pom),\,\, t_0 \in W^{-\frac{1}{2},2}(\pom), \,\,0\leq g \in L^\infty(\gs), \,\,a \in \R.$$
In addition, $\int_\Omega f + t_0=0$. We are looking for weak solutions $(u,u_c) \in W^{1,p}(\om) \times W^{1,2}_{loc}(\omc)$.\\

The above contact problem is equivalent to the following variational inequality in the space $$X^p = W^{1,p}(\om) \times \Wgs\ , \quad \Wgs = \{u \in W^{\frac{1}{2},2}(\pom) : \mathrm{supp}\ u \subset \bar{\Gamma}_s\}\ :$$

Find $(\hat u, \hat v) \in X^p$ such that for all $(u,v)\in X^p$,
\begin{align*}
&\langle A'(\nabla \hat u), \nabla u\rangle + \langle S(\hat u|_\pom + \hat v), u|_\pom\rangle = \int_\om f u +\langle t_0+Su_0,u|_\pom \rangle = L(u,0)\ ,\\
&\langle S(\hat u|_\pom + \hat v),v-\hat v \rangle + j(v)-j(\hat v)\geq \langle t_0 + S u_0, v-\hat v  \rangle = L(0,v-\hat v)\ .
\end{align*}
We obtain a variant of Galerkin orthogonality in the interior:
\begin{align*}
&\langle A'(\nabla\hat u)- A'(\nabla\hat u_h), \nabla u_h\rangle + \langle S((\hat u-\hat u_h)|_\pom + \hat v-\hat v_h), u_h|_\pom\rangle \\
& \qquad \qquad + \langle (S-S_h)(\hat u_h|_\pom +\hat v_h-u_0), u_h|_\pom\rangle= 0\ .
\end{align*}
As in \cite{scalar}, Theorem 2, the monotony of $A'$ and coercivity of $S$ imply
\begin{align*}
\|\hat u - \hat u_h, \hat v- \hat v_h\|_{X^p}^p &\lesssim |\hat u - \hat u_h|^2_{(1,\hat u, p)} + \|(\hat u-\hat u_h)|_\pom + \hat v-\hat v_h\|_{W^{\frac{1}{2},2}(\pom)}^2\\
&\lesssim \langle A'(\nabla\hat u)- A'(\nabla\hat u_h), \nabla (\hat u - \hat u_h)\rangle \\
& \quad + \langle S((\hat u-\hat u_h)|_\pom + \hat v-\hat v_h), (\hat u-\hat u_h)|_\pom + \hat v-\hat v_h\rangle\ .
\end{align*}
Using the variational equality in $\Omega$, the right hand side becomes
\begin{align*}
&\langle A'(\nabla\hat u)- A'(\nabla\hat u_h), \nabla (\hat u - \hat u_h)\rangle \\
& \quad + \langle S((\hat u-\hat u_h)|_\pom + \hat v-\hat v_h), (\hat u-\hat u_h)|_\pom + \hat v-\hat v_h\rangle\\
&= L(\hat u - \hat u_h,0) + \langle S(\hat u|_\pom + \hat v), \hat v - \hat v_h \rangle\\
&\quad -\langle A'(\nabla\hat u_h), \nabla (\hat u - \hat u_h)\rangle - \langle S(\hat u_h|_\pom +\hat v_h), (\hat u-\hat u_h)|_\pom\rangle\\
& \quad - \langle S(\hat u_h|_\pom +\hat v_h), \hat v-\hat v_h\rangle\ .
\end{align*}
Let $u_h\in W^{1,p}_h(\Omega)$ arbitrary and $(e, \tilde e) = (\hat u- \hat u_h, \hat v - \hat v_h)$, $e_h = \hat u- u_h$, whence $e-e_h = \hat u - u_h$.
With the help of Galerkin orthogonality in $\Omega$, the right hand side turns into
\begin{align*}
& L(e-e_h, 0) - \langle S (\hat u|_\pom + \hat v ), \tilde e\rangle - \langle A'(\nabla \hat u_h),\nabla (e-e_h)\rangle
 - \langle S(\hat u_h|_\pom + \hat v_h), e-e_h\rangle\\
 & \quad  + \langle S (\hat u_h|_\pom + \hat v_h ), \tilde e\rangle + \langle (S_h-S)(\hat u_h|_\pom +\hat v_h-u_0), e_h\rangle\ .
\end{align*}
Recall that $L(e-e_h, 0) = \int_\om f (e-e_h) +\langle t_0+Su_0,(e-e_h)|_\pom \rangle$. In \cite{scalar} it was shown for a suitable interpolant $e_h = \pi e$ and any $\varepsilon>0$,
$$\int_\om f(e-e_h) \lesssim \varepsilon |e|_{(1,\hat u, p)}^2+ C(\varepsilon) \eta_f^2 + \varepsilon \eta_{gr}^2\, $$
where \begin{align*}
\eta_{gr}^2 &= \sum_{K \in \mathcal{T}_h} \int_{K}G_{p,\delta}(\nabla \hat u_h, \nabla\hat u_h - G_h \hat u_h)\ ,\\
\eta_f^2 &= \sum_{K \in \mathcal{T}_h} \int_{K} G_{p',1}(|\nabla \hat u_h|^{p-1}, h_K (f-f_K))\ ,
\end{align*}
involve the gradient recovery resp.~the approximation error of $f$. Integrating by parts in the term
$- \langle A'(\nabla \hat u_h),\nabla (e-e_h)\rangle$ yields two terms,
$$-  \sum_{l \subset \pom} \int_l  \nu\cdot A'(\nabla \hat u_h)\ (e-e_h)$$
and
$$- \sum_{l \not\subset \pom } \int_l A_l (e-e_h) \lesssim \eta_{gr}^2 + \varepsilon (|e|_{(1, \hat u_h, p)}^2 + \eta_{gr}^2)\ .$$
Altogether we conclude
\begin{align*}
\|\hat u - \hat u_h, \hat v- \hat v_h\|_{X^p}^p &\lesssim |\hat u - \hat u_h|^2_{(1,\hat u, p)} + \|(\hat u-\hat u_h)|_\pom + \hat v-\hat v_h\|_{W^{\frac{1}{2},2}(\pom)}\\
&\lesssim 2\varepsilon |e|^2_{(1,\hat u,p)}+ C(\varepsilon)\eta_f^2 + (1+2\varepsilon)\eta_{gr}^2\\
& \quad + \langle \nu \cdot A'(\nabla \hat u_h) + S(\hat u_h|_\pom + \hat v_h-u_0)-t_0, \pi e - e\rangle\\
& \quad + \langle S((\hat u-\hat u_h)|_\pom + \hat v - \hat v_h), \hat v- \hat v_h \rangle\\
& \quad  + \langle (S_h-S)(\hat u_h|_\pom +\hat v_h-u_0), \pi e\rangle\ .
\end{align*}
We write the second--to--last term as $\langle \sigma(\hat u) - \sigma(\hat u_h), \hat v- \hat v_h\rangle$ and the friction conditions as $\sigma(\hat u) = -\zeta g$, $|\hat v| = \zeta \hat v$ for some $|\zeta|\leq 1$.
Then
\begin{align*}
&\langle \sigma(\hat u) - \sigma(\hat u_h), \hat v - \hat v_{h}\rangle\\
& = -\langle \zeta g, \hat v\rangle - \langle\sigma(\hat u_h), \hat v\rangle + \langle \zeta g, \hat v_{h}\rangle + \langle \sigma(\hat u_h),\hat v_{h}\rangle\\
& = -\langle g, |\hat v|\rangle - \langle\sigma(\hat u_h),\hat v\rangle + \langle \zeta g,\hat v_{h}\rangle + \langle \sigma(\hat u_h),\hat v_{h}\rangle\\
& \leq \langle (|\sigma(\hat u_h)|-g)_+,|\hat v|\rangle + \langle \zeta g,\hat v_{h}\rangle -\langle |\sigma(\hat u_h)|, |\hat v_{h}|\rangle+\langle |\sigma(\hat u_h)|, |\hat v_{h}|\rangle+ \langle \sigma(\hat u_h),\hat v_{h}\rangle\\
& \leq \langle (|\sigma(\hat u_h)|-g)_+,|\hat v-\hat v_h| + |\hat v_h|\rangle + \langle g,|\hat v_{h}|\rangle -\langle |\sigma(\hat u_h)|, |\hat v_{h}|\rangle+\langle |\sigma(\hat u_h)|, |\hat v_{h}|\rangle+ \langle \sigma(\hat u_h),\hat v_{h}\rangle\\
& \lesssim  \|(|\sigma(\hat u_h)|-g)_+\|_{\widetilde{W}^{-\frac{1}{2},2}(\gs)}\|\hat v- \hat v_{h}\|_{W^{\frac{1}{2},2}(\gs)} \\
&\qquad  + \langle (|\sigma(\hat u_h)|-g)_+ + g-|\sigma(\hat u_h)|, |\hat v_{h}|\rangle + \langle|\sigma(\hat u_h)|, |\hat v_{h}| \rangle +\langle \sigma(\hat u_h),\hat v_{h}\rangle\\
& =  \|(|\sigma(\hat u_h)|-g)_+\|_{\widetilde{W}^{-\frac{1}{2},2}(\gs)}  \|\hat v- \hat v_{h}\|_{W^{\frac{1}{2},2}(\gs)}\\
&\qquad + \int_\gs |(|\sigma(\hat u_h)|-g)_-| |\hat v_{h}| + 2 \int_\gs (\sigma(\hat u_h)\hat v_{h})_+\ .
\end{align*}
This proves the following a posteriori estimate:
\begin{thm}
Let $f \in L^{p'}(\om)$ and denote by $(e, \tilde e)$ the
error between the Galerkin solution $(\hat u_h, \hat v_h) \in X^p_h$ and the true solution $(\hat u , \hat v) \in X^p$. Then
\begin{eqnarray*}
\|\hat u - \hat u_h, \hat v - \hat v_h\|_{X^p}^p
&\lesssim& \eta_{gr}^2 + \eta_f^2 +\eta_S^2+\eta_\partial^2+\eta_g^2 ,
\end{eqnarray*}
where
\begin{align*}
\eta_{gr}^2 &= \sum_{K \in \mathcal{T}_h} \int_{K}G_{p,\delta}(\nabla \hat u_h, \nabla\hat u_h - G_h \hat u_h),\\
\eta_f^2 &= \sum_{K \in \mathcal{T}_h} \int_{K} G_{p',1}(|\nabla \hat u_h|^{p-1}, h_K (f-f_K)),\\
\eta_S^2 &= {\mathrm{dist}_{\Wmg} \left(V^{-1}(1-K)(\hat u +\hat v- u_0), \Whmg\right)}^2\\
\eta_{\partial}^2&=\|\nu \cdot A'(\nabla \hat u_h) + S(\hat u_h|_\pom + \hat v_h-u_0)-t_0\|_{W^{-1+\frac{1}{p},p'}}^{p'}\\
\eta_{g}^2 &= \|(|\sigma(\hat u_h)|-g)_+\|_{\widetilde{W}^{-\frac{1}{2},2}(\gs)}^{p'} + \int_\gs |(|\sigma(\hat u_h)|-g)_-| |\hat v_{h}| + \int_\gs (\sigma(\hat u_h)\hat v_{h})_+\ .
\end{align*}
\end{thm}

\begin{rem}
As $p\geq2$, we are here able to split both the discretized and the continuous variational inequality into an equation in $\Omega$ and an inequality on $\pom$. This explains the slightly different form of the frictional terms compared to Theorems 5.1 and 6.1.
\end{rem}

\end{document}